\documentclass[12pt]{amsart}
\pdfoutput=1
\usepackage{amscd,latexsym,amsthm,amsfonts,amssymb,amsmath,amsxtra,}
\usepackage[mathscr]{eucal}
\renewcommand\theequation{\thesection.\arabic{equation}}

\newcommand{\BA}{{\mathbb {A}}}

\newcommand{\BC}{{\mathbb {C}}}

\newcommand{\BQ}{{\mathbb {Q}}}
\newcommand{\BR}{{\mathbb {R}}}

\newcommand{\BZ}{{\mathbb {Z}}}

\newcommand{\CA}{{\mathcal {A}}}
\newcommand{\CB}{{\mathcal {B}}}
\newcommand{\CC}{{\mathcal {C}}}
\newcommand{\CE}{{\mathcal {E}}}
\newcommand{\CF}{{\mathcal {F}}}

\newcommand{\CH}{{\mathcal {H}}}

\newcommand{\CL}{{\mathcal {L}}}
\newcommand{\CM}{{\mathcal {M}}}
\newcommand{\CN}{{\mathcal {N}}}
\newcommand{\CO}{{\mathcal {O}}}

\newcommand{\CR}{{\mathcal {R}}}

\newcommand{\CW}{{\mathcal {W}}}

\newcommand{\Fa}{{\mathfrak {a}}}

\newcommand{\Fl}{{\mathfrak {l}}}
\newcommand{\Fm}{{\mathfrak {m}}}

\newcommand{\Fo}{{\mathfrak {o}}}
\newcommand{\Fp}{{\mathfrak {p}}}
\newcommand{\Fq}{{\mathfrak {q}}}

\newcommand{\RA}{{\mathrm {A}}}
\newcommand{\RB}{{\mathrm {B}}}

\newcommand{\RG}{{\mathrm {G}}}

\newcommand{\RK}{{\mathrm {K}}}

\newcommand{\RN}{{\mathrm {N}}}

\newcommand{\RX}{{\mathrm {X}}}

\newcommand{\RZ}{{\mathrm {Z}}}

\newcommand{\GL}{{\mathrm{GL}}}

\newcommand{\PGL}{{\mathrm{PGL}}}

\renewcommand{\Re}{{\mathrm{Re}}}

\newcommand{\SL}{{\mathrm{SL}}}

\newcommand{\bs}{\backslash}

\newcommand{\bqa}{\begin{equation}}
\newcommand{\eqa}{\end{equation}}
\newcommand{\bea}{\begin{eqnarray}}
\newcommand{\eea}{\end{eqnarray}}
\newcommand{\bna}{\begin{eqnarray*}}
\newcommand{\ena}{\end{eqnarray*}}
\newcommand{\bma}{\begin{pmatrix}}
\newcommand{\ema}{\end{pmatrix}}

\newtheorem{thm}{Theorem}[section]

\newtheorem{prop}[thm]{Proposition}

\newtheorem {ques/conj}[thm]{Question/Conjecture}

\newtheorem{rmk}[thm]{Remark}

\makeatletter

\newcommand{\Rmnum}[1]{\expandafter\@slowromancap\romannumeral #1@}
\makeatother

\begin{document}
\renewcommand{\theequation}{\arabic{equation}}
\numberwithin{equation}{section}

\title{Spectral Reciprocity for the product of Rankin-Selberg $L$-functions}

\author{Xinchen Miao}
\address{School of Mathematics\\
University of Minnesota\\
Minneapolis, MN 55455, USA}
\email{miao0011@umn.edu}

\date{October, 2021}

\begin{abstract}
We prove a new case of spectral reciprocity formulae for the product of $\GL(n+1) \times \GL(n)$ and $\GL(n) \times \GL(n-1)$ Rankin-Selberg $L$-functions ($n \geq 3$), which are first developed by Blomer and Khan in \cite{BK17} for degree 8 $L$-functions ($n=2$ case, the product of $\GL(3) \times \GL(2)$ and $\GL(2) \times \GL(1)$ Rankin-Selberg $L$-functions). Our result can be viewed as a generalization of Blomer and Khan's work to higher rank case. We will mainly follow the method developed in \cite{Nun20}. We will use the integral representations of Rankin-Selberg $L$-functions generalized by Ichino and Yamana \cite{IY15}, spectral theory of $L^2$ space and the language of automorphic representations.

\end{abstract}

\maketitle

\tableofcontents

\section{Introduction}

In the recent five years, people have shown a lot of interests in studying so called automorphic spectral reciprocity formulae (For example, \cite{AK18} \cite{BK17} \cite{BK18} \cite{BML19} \cite{Nel19} \cite{Nun20} \cite{Zac19} \cite{Zac20} etc.), which we mean an identity roughly of the shape as follows:
$$ \sum_{\pi \in \CF} \CL(\pi)\CH(\pi)= \sum_{\pi \in \tilde{\CF}} \widetilde{\CL}(\pi) \widetilde{\CH}(\pi),$$
where $\CF$ and $\widetilde{\CF}$ are different families of automorphic representations, $\CL(\pi)$ and $\widetilde{\CL}(\pi)$ are certain (may different) (products of) $L$-functions corresponding to automorphic representations (see \cite{Mot93}). Moreover, $\CH$ and $\widetilde{\CH}$ are some (global) weight functions related to automorphic representation $\pi$. The map from $\CH$ to $\widetilde{\CH}$ is given by an explicit integral transform.

Besides the beauty of spectral reciprocity formulae, these kind of formulae always have powerful applications to non-vanishing and subconvexity problems for the associated $L$-functions, which is very important in analytic number theory. This gives people motivations to understand the intrinsical symmetry behind such kind of identities. In other word, people will ask whether there exists a master formula to cover all these reciprocity formulae related to $L$-functions. Moreover, there are several works on automorphic $L$-functions on which one can find such spectral reciprocity formulae hidden as the key ingredients inside the proof (For example, see \cite{KY21} and \cite{Kha21}).

This kind of spectral reciprocity formula in higher rank (degree 8 $L$-function) first appeared in a paper by Blomer, Miller and Li \cite{BML19}. They established a spectral reciprocity formula for $\GL(4) \times \GL(2)$ Rankin-Selberg $L$-functions in the $\GL(2)$ archimedean (spectral) aspect. Later, Blomer and Khan's work in \cite{BK17} shows another kind of interesting and deep spectral reciprocity formula which contain the product of $\GL(3) \times \GL(2)$ Rankin-Selberg $L$-functions and the standard $\GL(2)$ $L$-functions. This is a degree 8 $L$-function case again. In \cite{BK17}, Blomer and Khan used powerful analytic number theory tools to go further. They showed that if the input of $\GL(3)$ automorphic form is an Eisenstein series other than a cusp form, the spectral reciprocity formula is still true. Therefore, they may successfully to apply this kind of spectral reciprocity formula to study the $\GL(2)$ subconvexity problem in level aspect. However, such kind of spectral reciprocity identities appeared at least since Motohashi's formula \cite{Mot93} which connects the fourth moment of the Riemann zeta-function ($\GL(1)$ $L$-functions) to the cubic moment of standard $L$-functions of cusp forms on $\GL(2)$.

Now we briefly describe the main result in Blomer and Khan's paper. They roughly say the following:

Let $F$ be an automorphic form (cusp form or Eisenstein series) for the group $\SL_3(\BZ)$. In the following summations, we let $f$ be a cusp form for a congruence subgroup of $\SL_2(\BZ)$ ($f$ is ramifed at some finite places). They proved a spectral reciprocity formula of the following shape
$$\sum_{f \text{ of level } q} L(s, F \times f) L(w, f) \lambda_f(\ell) \rightsquigarrow \sum_{f \text{ of level } \ell} L(s', F \times f) L(w', f) \lambda_f(q),$$
where $q$ and $\ell$ are coprime integers and $\lambda_f$ is the corresponding Hecke eigenvalue for the cusp form $f$ and 
\begin{equation}
\textstyle  s' = \frac{1}{2}(1+w-s), \quad w' = \frac{1}{2}(3s+w-1).
\end{equation}

Furthermore, if $F$ is an Eisenstein series, we will also see that
$$\sum_{f \text{ of level } q}   \vert L(1/2, f)  \rvert^4 \lambda_f(\ell) \rightsquigarrow \sum_{f \text{ of level } \ell}  \vert L(1/2, f) \rvert^4 \lambda_f(q).$$

In 2020, Nunes \cite{Nun20} gave a new and nice proof of Blomer-Khan's result \cite{BK17} and generalized their result to general number field instead of $\BQ$. Instead of tools in the analytic number theory (such as Kuznetsov trace formula, Voronoi Summation formula etc.), his main method is the spectral decomposition formula and the integral representation of Rankin-Selberg $L$-functions. The reciprocity of two coprime ideal $\Fp$ and $\Fl$ are given by the action of a Weyl element. We will follow Nunes' method. Applying integral representations of Rankin-Selberg $L$-functions generalized by Ichino and Yamana \cite{IY15}, spectral decomposition theory and the language of automorphic representations, we will generalize Blomer-Khan and Nunes' result to arbitrary higher rank groups. We will establish a spectral reciprocity formula for the product of $\GL(n+1) \times \GL(n)$ and $\GL(n) \times \GL(n-1)$ ($n \geq 3$) Rankin-Selberg $L$-functions, which is a degree $2n^2$ $L$-functions. Note that if $n=2$, the result has already been established in \cite[Theorem 1.1, 1.2]{Nun20}. Roughly speaking, we will prove a spectral reciprocity formula of the shape:
$$\sum_{\pi \text{ of level } \Fq} L(s, \Pi \times \widetilde{\pi} ) L(w, \pi \times \pi_1) H(\pi) \rightsquigarrow \sum_{\pi \text{ of level } \Fl} L(s', \Pi \times \widetilde{\pi}) L(w', \pi \times \pi_1) \widetilde{H}(\pi)$$
for different unramified prime ideal $\Fq$ and $\Fl$, where $\widetilde{\pi}$ is the contragredient representation of $\pi$.
Here both $\Pi$ and $\pi_1$ are cuspidal automorphic forms on $\GL(n+1)$ and $\GL(n-1)$ which are unramified everywhere and have trivial central character. Note that the automorphic representation $\pi$ on $\GL(n)$ also have trivial central character. Moreover, we have
$$s' = \frac{1+(n-1)w-s}{n}, \quad w' = \frac{(n+1)s+w-1}{n}.$$
For the rigorous statement of above spectral reciprocity formula, the readers can see Theorem \ref{spectral-reciprocity}.

This kind of spectral reciprocity formula should have some applications on the simultaneous nonvanishing problems of certain Rankin-Selberg $L$-functions \cite[Corollary 1.3]{Nun20}. However, this need a careful choice of local vectors and estimations. Moreover, we have to bound all the error terms which contain the moment of higher rank $L$-functions. We hope to go back to this kind of applications in the future.

\section{Some notations and preliminaries}

We will mainly follow the notations and preliminaries in \cite{Ja20}, \cite{Nun20}, and \cite{Zac19}.

\subsection{Number Fields and Local Fields}

Let $F/\mathbb{Q}$ be a fixed number field with ring of intergers $\CO_F$ and discriminant $\Delta_F$.

For a place $v$ of $F$, we let $F_v$ be a local field which is the completion of $F$ at the place $v$. If $v$ is non-Archimedean, we write $\Fo_v$ for the ring of integers in $F_v$ with maximal ideal $\Fm_v$ and uniformizer $\varpi_v$. The cardinality of the residue field is $p_v:=\vert \Fo_v /\Fm_v \rvert$. For $s\in\BC$, we define the local zeta function $\zeta_{F_v}(s)$ to be $(1-p_v^{-s})^{-1}$ if $v$ is a finite place; $\zeta_{F_v}(s)=\pi^{-s/2}\Gamma(s/2)$ if $v$ is real and $\zeta_{F_v}(s)=2(2\pi)^{-s}\Gamma(s)$ if $v$ is complex.

The adele ring of $F$ is denoted by $\BA_F$ and its unit group is given by $\BA^\times_F$ (idele group). We also let $\BA^1_F:=\{ x\in \BA_F^\times \ : \ |x|=1\}$, where $|\cdot| : \BA_F^\times \rightarrow \BR_{>0}$ is the adelic norm map. Note that $\BA_F^1$ is exactly the kernel of the adelic norm map. We also call $\BA^1_F$ norm one ideles.

We fix $\psi = \prod_v \psi_v$ be the additive character with the form as $\psi_\BQ \circ \text{Tr}_{F/\BQ},$ where $\text{Tr}_{F/\BQ}$ is the trace map and $\psi_\BQ$ is the standard additive character on $\BQ \setminus \BA_\BQ$. For $v<\infty$, we let $d_v$ be the conductor of additive character $\psi_v$, which is the smallest non-negative integer such that $\psi_v$ is trivial on $\Fm_v^{d_v}$. In this case, we will have $\Delta_F=\prod_{v<\infty}q_v^{d_v}$. We may also set $d_v=0$ when $v$ is Archimedean.

\subsection{Subgroups of $\GL(n)$ and Measure}

Now we consider some subgroups of $\GL(n)$.

Let $\RG=\GL(n)$. If $R$ is a field, by definition, $\RG(R)$ is the group of $n \times n$ matrices with coefficients in $R$ and determinant in $R-\{0\}$. We also define the following standard and also important subgroups
$$ \RZ_n(R):= \left\{ \begin{pmatrix} u &&& \\ & u && \\ && \cdots & \\ &&& u \end{pmatrix}, \; \; u \in R^* \right \},$$
$$ \RN_n(R):= \left\{ \begin{pmatrix} 1 & x_{1,2} & x_{1,3} & \cdots & x_{1,n} \\ & 1 & x_{2,3} & \cdots & x_{2,n} \\ && \cdots & \cdots & \cdots \\ &&& 1 & x_{n-1,n} \\ &&&& 1 \end{pmatrix},\;\; x_{i,j} \in R,\; 1 \leq i<j \leq n \right \},$$
$$ \RA_n(R):= \left\{ \begin{pmatrix} y_{n-1} \cdots y_1 &&&& \\ & y_{n-2}\cdots y_1 &&& \\ && \cdots && \\ &&& y_1 & \\ &&&& 1 \end{pmatrix}, \;\; y_i \in R^*, \; 1 \leq i \leq n-1 \right \},$$
and the Borel subgroup
$$ \RB_n(R):=  \RZ_n(R)\RN_n(R)\RA_n(R).$$
We let $W$ be the Weyl group of $\RG$. We write
$$ w_0:= \begin{pmatrix} &&&& 1 \\ &&& 1 & \\ && \cdots && \\ & 1 &&& \\ 1 &&&& \end{pmatrix} $$
be the longest Weyl element.

Moreover, for any place $v$, we let $\RK_v$ be the maximal compact subgroup of $\RG(F_v)$ which is defined by
\begin{equation}\label{Compact}
\RK_v= \left\{ \begin{array}{lcl}
\RG(\Fo_v) & \text{if} & v \ \mathrm{is \ finite} \\
 & & \\
\mathrm{O}_n(\BR) & \text{if} & v \ \mathrm{is \ real} \\
 & & \\
\mathrm{U}_n(\mathbb{C}) & \text{if} & v \ \mathrm{is \ complex}.
\end{array}\right.
\end{equation}
We also let $\RK:= \prod_v \RK_v$. If $v<\infty$ and $n\geqslant 0$, we define the congruence subgroup \cite{JPS81} by (This will be used in the computation of local vectors and is useful in the local new-vector theory)
$$\RK_{v,0}(\varpi_v^n)= \left\{ \begin{pmatrix} k_{1,1} & k_{1,2} & \cdots & k_{1,n} \\ k_{2,1} & k_{2,2} & \cdots & k_{2,n} \\ \cdots & \cdots & \cdots & \cdots \\ k_{n,1} & k_{n,2} & \cdots & k_{n,n} \end{pmatrix} \in \RK_v, \; \;  k_{n,i} \in \Fm_v^n,\; 1 \leq i \leq n-1 \right\}.$$
If $\Fa$ is an ideal of $\CO_F$ with prime decomposition $\Fa=\prod_{v<\infty}\Fp_v^{f_v(\Fa)}$ where $\Fp_v$ is the unique prime (also maximal) ideal corresponding to the finite place $v$, then we define
$$\RK_0(\Fa):=\prod_{v<\infty} \RK_{v,0}(\varpi_v^{f_v(\Fa)}) \subseteq K \subseteq \GL_n(\BA_F).$$
Now we have to normalize the measures we need.

At each place $v$, $dx_v$ denotes a self-dual measure on $F_v$ with respect to the standard additive character $\psi_v$. If $v<\infty$, $dx_v$ gives a Haar measure on $F_v$ which gives the volume $q_v^{-d_v/2}$ to the integer ring $\Fo_v$. If $v$ is real, the measure $dx_v$ is the standard Lebesgue measure on $\BR$. If $v$ is complex, the measure $dx_v$ is the multiplication of two and the standard Lebesgue measure on $\BC$. We define $dx:=\prod_v dx_v$ on $\BA_F$. Moreover, we will take $d^\times x_v:=\zeta_{F_v}(1)\frac{dx_v}{|x_v|}$ as the Haar measure on the multiplicative group $F_v^\times$ and we let $d^\times x := \prod_v d^\times x_v$ be the Haar measure on the idele group $\BA^\times_F$.

\vspace{0.1cm}

We equip $\RK_v$ with the probability Haar measure $dk_v$. In other word, the volume of $\RK_v$ equals to one.

Using the Iwasawa decomposition on $\RG(F_v)$ which gives $\RG(F_v)=\RZ_n(F_v)\RN_n(F_v)\RA_n(F_v)\RK_v$, a Haar measure on $\RG(F_v)$ can be given by 
\begin{equation}\label{measure}
dg_v = d^\times u \prod_{1 \leq i<j \leq n} dx_{i,j} \frac{\prod_{i=1}^{n-1} d^\times y_i}{\delta(\RA_n)} dk_v,
\end{equation} 
where $\delta(A_n)=\delta(y_1, y_2, \cdots, y_{n-1})$ is the modular character defined on $\RA_n(F_v)$.
The measure on the adelic points of the subgroups in $\GL(n)$ are just the product of the local measures defined above. We also denote by $dg$ the quotient measure on the space $$\RX:= \RZ_n(\BA_F)\RG(F) \bs \RG(\BA_F),$$ 
with total volume $\mathrm{vol}(\RX)<\infty$.

\subsection{Whittaker functions}

We recall some basic background of Whittaker functions.

We start from the generic representations. Let $\pi$ be a global genreic automorphic representation of $\GL_n(\BA_F)$ and let $\phi\in \pi$ be a generic automorphic form. Let $W_{\phi}:\GL_n(\BA_F)\rightarrow \BC$ be the Whittaker function of $\phi$ which is given by
\begin{equation} \label{Whittaker-Cuspidal}
W_{\phi}(g):=\int_{\RN_n(F)\bs \RN_n(\BA_F)}\phi\left(n g\right)\psi_1^{-1}(n) dn.
\end{equation}
We will give the definition of the additive character $\psi_1$ later.

Since $\pi$ is generic, the Whittaker function does not vanish. The integral defined above \eqref{Whittaker-Cuspidal} is absolute convergent since the integral domain $\RN_n(F)\bs \RN_n(\BA_F)$ is compact and the integral is moreover uniformly comvergent on any compact subsets in $\GL_n(\BA_F)$. 

By changing variables, we note that $W_{\phi}(ng)=\psi_1(n)W_{\phi}(g)$ for all $n\in \RN_n(\BA_F)$, where $\psi_1$ is a nontrivial standard additive character from $\RN_n(F) \bs \RN_n(\BA_F)$ to $\BC$. In our paper, we will choose that $\psi_1(n):=\psi \left( \sum_{i=1}^{n-1} x_{i,i+1} \right)$, where $n=(x_{i,j}) \in \RN_n(\BA_F)$. Here $\psi$ is a standard additive character from $\BA_F / F$ to $\BC$ which is defined in Section 2.1.

We also have the following Fourier series expansion (see \cite[Theorem 1.1]{Cog07}) when we further assume that $\phi$ is a cusp form. Therefore, a cusp form is automatically generic by the following equation.

\begin{equation} \label{Fourierexpansion}
\phi(g)=\sum_{\gamma\in \RN_{n-1}(F)\bs \GL_{n-1}(F)}W_{\phi} \left( \begin{pmatrix} \gamma&\\&1 \end{pmatrix} g \right).
\end{equation}

We can also define local Whittaker functions for smooth generic admissible representations of $\GL_n$ over local fields $F_v$. We have the following well known decomposition theorem. If $\pi$ is a generic smooth irreducible admissible representation of $\GL_n(\BA_F)$, then we know that $\pi$ factors as a restricted tensor product by $\pi\simeq\otimes'_v\pi_v$. For each $v$, $\pi_v$ is a local generic smmoth irreducible admissible representation of $\GL_n(F_v)$. For each local place $v$, we can define local Whittaker functions. For every $\phi\in \pi$, if we write $\phi=\otimes'_v\phi_v$, then we have the decomposition

\begin{equation}
W_{\phi}(g)=\prod_v W_{\phi_v}(g_v),\; \text{where} \; g=(g_v)_v\in \GL_n(\BA_F).
\end{equation}

In fact, we know that the map $\phi\mapsto W_{\phi}$ intertwines the space of $\pi$ and the space
$$
\CW(\pi,\psi_1):=\{W_{\phi};\;\phi\in\pi\},
$$
which is called the Whittaker model of $\pi$. We can similarly define the local Whittaker models $\CW(\pi_v,\psi_{1,v})$. 

It is also important to consider Whittaker functions with respect to the character $\psi_1'=\psi_1^{-1}$ since it will appear in the integral representations for $\GL(n+1) \times \GL(n)$ Rankin-Selberg $L$-functions (see Section 3.2). It also has close relation with the Whittaker model of the contragredient representation $\widetilde{\pi}$ when the representation $\pi$ is unitarizable (see \cite[Appendix A]{FLO12}). Locally, we can define $W'_{\phi_v}$ by simply replacing $\psi_{1,v}$ by $\psi'_{1,v}=\psi_{1,v}^{-1}$. Globally, we can define $W'_{\phi}$ by replacing $\psi_1$ by $\psi_1'$. Moreover, we have the following equation

\begin{equation}
W'_{\overline{\phi_v}}=\overline{W_{\phi_v}}\text{ for all places $v$ of $F$}.
\end{equation}

For $W \in \CW(\pi_v,\psi_{1,v})$, the corresponding Whittaker model $\CW(\widetilde{\pi}_v,\psi_{1,v}^{-1})$ for contragredient representation is given by $\widetilde{W}(g):= W \left(w_0 (g^t)^{-1} \right)$, where $w_0$ is the longest Weyl element in $\GL(n)$ and $(g^t)^{-1}$ means the transpose inverse, therefore it is an involution in $\GL(n)$. The global Whittaker model $\CW(\widetilde{\pi},\psi_1^{-1})$ is defined in the same way. If a local generic smooth irreducible admissible representation $\pi_v$ is unramified, this is equivalent to say that there exists a vector which is right invariant by the action of the maximal compact subgroup of $\RG(F_v)$ in the space of $\pi$. We call such a vector spherical vector and spherical vectors are unique up to multiplication by scalars. We say that a vector $\phi_v\in\pi_v$ is normalized spherical if it is spherical. Moreover, its related Whittaker function $W_{\phi_v}$ satisfies $W_{\phi_v}(I_{n \times n})=1$.

\subsection{Spectral Decomposition}

Now we consider the $L^2$ space $L^2(\RX)$ which is the Hilbert space of complex valued square integrable functions with the domain $\RX$. The $L^2$-norm is defined by
\begin{equation}\label{L^2norm}
||\varphi||_{L^2(\RX)}^2 := \int_\RX |\varphi (g)|^2dg.
\end{equation}
For any $\varphi_1,\varphi_2 \in L^2(\RX)$, we have
$$ ( \varphi_1, \varphi_2)_{L^2(\RX)}:= \int_\RX \varphi_1(g) \overline{\varphi_2(g)} dg.$$

An important closed subspace of $L^2(\RX)$ is the closed subspace of cusp forms. We let  $L_{\mathrm{cusp}}^2(\RX)$ be the closed subspace of cusp forms. A cusp form is the function $\varphi\in L^2(\RX)$ with the following additional equation
$$\int_{\RN_n(F) \setminus \RN_n(\BA_F)}\varphi(ng)dn=0, $$
for almost all $g \in \GL_n(\BA_F)$.

The group $\RG(\BA_F)$ acts by right translations on both spaces $L^2(\RX)$ and $L_{\mathrm{cusp}}^2(\RX)$ and the corresponding representation is unitary with respect to the norm in \eqref{L^2norm}. The decomposition theorem of automorphic representations is well known, which states that each irreducible component $\pi$ factors as a restricted tensor product $\pi = \otimes'_v \pi_v$ for all places $k$, where $\pi_v$ are irreducible and unitary representations of the local groups $\RG(F_v)$. The spectral decomposition is established in 1970s by Gelbart, Jacquet and Langlands which gives the following orthogonal decomposition (see \cite{MW95} for more details)
\begin{equation}
L^2(\RX)=L^2_{\mathrm{cusp}}(\RX)\oplus L^2_{\mathrm{res}}(\RX)\oplus L^2_{\mathrm{cont}}(\RX).
\end{equation}
Here the closed subspace $L^2_{\mathrm{cusp}}(\RX)$ decomposes as a direct sum of irreducible $\RG(\BA_F)$-representations which are called the cuspidal automorphic representations. $L^2_{\mathrm{res}}(\RX)$ is called the residue spectrum which is the direct sum of all residue automorphic representations of $L^2(\RX)$. Finally, the continuous part $L^2_{\mathrm{cont}}(\RX)$ is a direct integral of irreducible $\RG(\BA_F)$-representations and it is expressed by the Eisenstein series. 

\vspace{0.1cm}

For any ideal $\Fa$ of $\CO_F$, we let $L^2(\RX,\Fa):= L^2(\RX)^{\RK_0(\Fa)}$ be the closed subspace of level $\Fa$ automorphic forms. This is the closed subspace of $L^2$ functions that are invariant under the subgroup $\RK_0(\Fa)$.

We now have the following spectral orthogonal decomposition with the level restriction:
\begin{equation} \label{Decomposition}
L^2(\RX, \Fa)=L^2_{\mathrm{cusp}}(\RX, \Fa)\oplus L^2_{\mathrm{res}}(\RX, \Fa)\oplus L^2_{\mathrm{cont}}(\RX, \Fa).
\end{equation}

\subsection{Automorphic Representations}

Now we consider automorphic representations. Let $\hat{\RX}$ be the isomorphism class of unitary irreducible automorphic representations which will appear in the spectral decomposition of $L^2(\RX)$. Since we will later use the integral representation of $\GL(n+1) \times \GL(n)$ Rankin-Selberg $L$-functions, we will only focus on the unitary irreducible generic automorphic representations. We will consider $\hat{\RX}_{\text{gen}}$ be the subset of $\hat{\RX}$ which is the isomorphism class of generic representations in $\hat{\RX}$, which is the unitary irreducible automorphic representation class that have (unique) Whittaker models. We fix an automorphic Plancherel measure $d \mu_{\text{aut}}$ on $\hat{\RX}$ which is compatible with the Haar measure on $\RX$.

Fortunately, we have the Langlands classification for $\hat{\RX}_{\text{gen}}$ (see \cite{CPS94} and \cite{MW95}). We take a partition $n=r_1+\dots+r_k$. Let $\pi_j$ be a unitary cuspidal automorphic representation for $\GL_{r_j}(\BA_F)$ (If $r_j=1$, we simply take $\pi_j$ to be a unitary Hecke character). We now consider the unitary induced representation $\Pi$ from the Levi subgroup $\GL(r_1)\times\dots\times\GL(r_k) \subseteq \GL(n)$ to $\GL(n)$ with the tensor product representation $\pi_1\otimes\dots\otimes\pi_k$. There exists a unique irreducible constituent of $\Pi$ which we denote by the isobaric sum $\pi_1\boxplus\dots\boxplus\pi_k$. Then, Langlands classification says that every element in $\hat{\RX}_\text{gen}$ is isomorphic to such an isobaric sum. 

Moreover, we know that all residual automorphic representations in $L^2_{\mathrm{res}}(\RX)$ is not generic (see \cite{JL13} Proposition 2.1).

Now we note that all the unitary generic Eisenstein series will have the form of the isobaric sum $\pi_1 \vert \cdot \rvert^{i t_1} \boxplus\dots\boxplus\pi_k \vert \cdot \rvert^{i t_k}$, where $t_1,t_2,\cdots,t_k \in \BR$. We will write
$$ \pi= \pi(\pi_1,\cdots,\pi_k,t_1,\cdots,t_k)=\pi_1 \vert \cdot \rvert^{i t_1} \boxplus\dots\boxplus\pi_k \vert \cdot \rvert^{i t_k}$$
with parameters $t_1,\cdots,t_k$ ($k \geq 2$).

\section{The product of $\GL(n+1) \times \GL(n)$ and $\GL(n) \times \GL(n-1)$ ($n \geq 3$) Rankin-Selberg $L$-functions}

\subsection{Abstract Spectral Decomposition Formula}

We have the following abstract spectral decomposition formula in $\GL(n)$ case (see \cite[Section 2.2]{MV10}). The discrete part is generated by cusp forms and residue representations. The continuous spectrum part which is expressed by Eisenstein series is complicated. However, we know that it depends on the partition of the positive integer $n$ (see Section 2.5).

\begin{prop}  \label{spectral theorem}
Suppose that $\RX:= \RZ_n(\BA_F) \GL_n(F) \bs \GL_n(\BA_F)$. Let $F\in C^{\infty}\left(\RX / \RK_0(\Fa)\right)$ and of rapid decay, then we have the following equation:
\begin{equation}
\begin{aligned}
F(g) &= \int_{\substack{\pi \in \hat{\RX} \\ \mathrm{cond}(\pi) |  \Fa}}\left( \sum_{\phi \in \CB(\pi,\Fa)} \langle  F, \phi \rangle \phi(g) \right) d \mu_{\mathrm{aut}}(\pi), \\
&= \int_{\substack{\pi \ \mathrm{automorphic} \\ \mathrm{cond}(\pi) |  \Fa}}\left( \sum_{\phi \in \CB(\pi,\Fa)} \langle  F, \phi \rangle \phi(g) \right) d \mu_{\mathrm{aut}}(\pi), \\
&=  \sum_{\substack{\pi \ \mathrm{cuspidal} \\ \mathrm{cond}(\pi) |  \Fa}}\sum_{\phi \in \CB(\pi,\Fa)} \langle  F, \phi \rangle \phi(g)+ \sum_{\substack{\pi \ \mathrm{residue} \\ \mathrm{cond}(\pi) |  \Fa}}\sum_{\phi \in \CB(\pi,\Fa)} \langle  F, \phi \rangle \phi(g) \\
&+  \int_{\substack{\pi \ \mathrm{continuous} \\ \mathrm{cond}(\pi) |  \Fa}}\left( \sum_{\phi \in \CB(\pi,\Fa)} \langle  F, \phi \rangle \phi(g) \right) d \mu_{\mathrm{aut}}(\pi).
\end{aligned}
\end{equation}
\end{prop}

\begin{rmk}
Above Proposition \ref{spectral theorem} is a smooth version of the spectral decomposition \ref{Decomposition}.
 
\end{rmk}

\subsection{Rankin-Selberg $L$-functions}

We need to recall integral representations of $\GL(n+1) \times \GL(n)$ ($n \geq 3$) $L$-functions.

The theory is quite similar to the adelic Hecke-Jacquet-Langlands' theory \cite{God18} of twisted $L$-functions for $\GL_2 \times \GL_1$. We let $\Pi$ be irreducible automorphic representations of $\GL_{n+1}(\BA)$. Let $\Phi \in \Pi$ be an automorphic form. Let $\pi$ be irreducible automorphic representations of $\GL_n(\BA)$ and let $\phi\in \pi$ be an automorphic form. We first assume that $\Phi$ is a cusp form and is of rapid decay. Therefore, for every $s \in \BC$, we can consider the following period integral

$$
I(s,\Phi,\phi):=\int_{\GL_n(F)\bs \GL_n(\BA)}\Phi\begin{pmatrix} h&\\&1 \end{pmatrix} \phi(h)|\det h|^{s-\frac12} dh,
$$
which defines an entire function of $s$ and is bounded on vertical strips.

From the Whittaker-Fourier expansion of cusp forms \eqref{Fourierexpansion} \cite[Theorem 1.1]{Cog07}, if $\Phi$ is a cusp form, we will have (for $\Re(s)$ large enough)
\begin{equation}
I(s,\Phi,\phi)=\Psi(s,W_{\Phi},W'_{\phi}),\;\;(\Re(s)\gg 1),
\end{equation}
where the global zeta integral is given by
\begin{equation}
\Psi(s,W_{\Phi},W'_{\phi}):=\int_{N_n(\BA)\bs \GL_n(\BA)}W_{\Phi} \begin{pmatrix} h&\\&1 \end{pmatrix} W'_{\phi}(h)|\det h|^{s-\frac12} dh.
\end{equation}
The following result can be found in \cite{Cog07} \cite{JPS79} \cite{JPS83} \cite{JS90} \cite{Jac09}.

\begin{prop} 
Let $\Phi=\otimes'_v\Phi_v\in \Pi$ and $\phi=\otimes'_v\phi_v\in \pi$ be factorizable automorphic forms on $\GL_{n+1}(\BA_F)$ and $\GL_n(\BA_F)$. Let $W_{\Phi_v}$ and $W'_{\phi_v}$ be the corresponding Whittaker functions defined in Section 2.3. Then, for $\Re(s)$ large enough, the global zeta integral $\Phi(s,W_{\Phi},W'_{\phi})$ converges and we have the following factorization (Euler product)
$$
\Psi(s,W_{\Phi},W'_{\phi})=\prod_v\Psi_v(s,W_{\Phi_v},W'_{\phi_v}),
$$
where the local zeta integral is given by
\begin{equation}
\Psi_v(s,W_{\Phi_v},W'_{\phi_v}):=\int_{N_n(F_v)\bs \GL_n(F_v)}W_{\Phi_v} \begin{pmatrix} h_v&\\&1 \end{pmatrix} W'_{\phi_v}(h_v)|\det h_v|^{s-\frac12} d h_v.
\end{equation}
Moreover, if both $\Pi_v$ and $\pi_v$ are unramified and $\Phi_v$ and $\phi_v$ are normalized spherical vectors, we will have 
$$
\Psi_v(s,W_{\Phi_v},W'_{\phi_v})=L(s,\Pi_v\times\pi_v).
$$
\end{prop}

One of the key ingredients in our paper is the following generalization of integral representation on Rankin-Selberg $L$-functions for $\GL(n+1) \times \GL(n)$ given by Ichino and Yamana \cite{IY15}. 

\begin{prop} \label{IY}
Let $\Phi \in \Pi$ and $\phi \in \pi$ be automorphic forms on $\GL_{n+1}(\BA_F)$ and $\GL_n(\BA_F)$. Assume that the following period integral
$$
I(s,\Phi,\phi):=\int_{\GL_n(F)\bs \GL_n(\BA)}\Phi\begin{pmatrix} h&\\&1 \end{pmatrix} \phi(h)|\det h|^{s-\frac12} dh
$$
is absolute convergent, we will still have the following equation for $\Re(s)$ large enough:
\begin{equation}
I(s,\Phi,\phi)=\Psi(s,W_{\Phi},W'_{\phi}),\;\;(\Re(s)\gg 1),
\end{equation}
where
\begin{equation}
\Psi(s,W_{\Phi},W'_{\phi}):=\int_{N_n(\BA)\bs \GL_n(\BA)}W_{\Phi} \begin{pmatrix} h&\\&1 \end{pmatrix} W'_{\phi}(h)|\det h|^{s-\frac12} dh.
\end{equation}
\end{prop}

\begin{proof}
This is a combination of Corollary 3.10 and Main Theorem (Theorem 1.1) of Ichino and Yamana's paper \cite{IY15}.
\end{proof}

\subsection{Abstract Reciprocity Formula}

We can briefly summarize the proof of Theorem \ref{pre} now. The proof of Theorem \ref{pre} is a combination of above decomposition formula \ref{spectral theorem} and an identity between two periods. Later, we will relate the period to moments of certain $L$-functions. The proof of an identity between two periods is a rather simple matrix computation and is really important to abstract pre-spectral reciprocity formula.

Suppose that $\Phi\in C^{\infty}(Z_{n+1}(\BA)\GL_{n+1}(F)\bs\GL_{n+1}(\BA))$ is a cuspidal automorphic form. Therefore, it is of rapid decay.

Then, we can define the projection by
\begin{equation} 
\CA_s\Phi(h_n) \,:=|\det h|^{s-\frac{1}{2}} \cdot \int_{F^\times\bs\BA^{\times}}  \Phi\left(\begin{pmatrix} z_n(u)h_n&\\&1 \end{pmatrix} \right)|u|^{n(s-\frac{1}{2})}d^{\times} u.
\end{equation}

Since $\Phi$ is of rapid decay, the above average projection map over the center $\CA_s \Phi$ is well-defined for every complex number $s$ and is again of rapid decay in terms of $h_n \in \GL(n)$. Moreover, we may easily check that $\CA_s \Phi$ is invariant under the action of the center $\RZ_n(\BA)$.

We also give the following definition of the period

$$
I(w,\phi,\varphi):=\int_{\GL_{n-1}(F)\bs \GL_{n-1}(\BA_F)}  \phi\begin{pmatrix} h_{n-1}&\\&1 \end{pmatrix} \varphi(h_{n-1})|\det h_{n-1}|^{w-\frac12} dh_{n-1},
$$

whenever it is converges. Here $\phi$ is an automorphic form defined on $\GL_n(\BA)$ and $\varphi$ is a fixed everywhere unramified automorphic form for $\GL_{n-1}(\BA)$. Moreover, we assume that $\varphi$ is invariant under the center $Z(\BA)$, therefore its central character is trivial. If $n=2$, we see that $\varphi=1$ which is the trivial character. We can now state the abstract spectral reciprocity formula.

\begin{prop} \textbf[Abstract Reciprocity Formula] \label{abstract-reciprocity-formula}

Let $\Phi\in C^{\infty}(Z_{n+1}(\BA)\GL_{n+1}(F)\bs\GL_{n+1}(\BA))$ be a cusp form. Then, for every $s,\,w\in\BC$, we have the following abstract reciprocity equation

$$
I(w,\CA_s\Phi,\varphi)=I(w',\CA_{s'}\check{\Phi},\varphi),
$$
where $s'=\frac{1+(n-1)w-s}{n},\,w'=\frac{(n+1)s+w-1}{n}$, and $\check{\Phi}$ is given by 
\begin{equation}\label{check-Phi}
\check{\Phi}=\Pi\left(w_{12} \right)\cdot\Phi,\;\;w_{12}:=\begin{pmatrix} I_{n-1} &&\\&&1\\&1& \end{pmatrix}.
\end{equation}

\end{prop}

\begin{proof}
From the definition, we may write

\begin{equation} \label{As}
\begin{aligned}
I(w,\CA_s\Phi,\varphi) &= \int_{F^{\times}\bs\BA^{\times}} \int_{\GL_{n-1}(F)\bs \GL_{n-1}(\BA_F)} \Phi \left( \begin{pmatrix} z_{n-1}(u)h_{n-1} && \\ & u & \\ && 1 \end{pmatrix} \right) \varphi(h_{n-1}) \\
& \cdot \vert u \rvert^{n(s-\frac{1}{2})} \vert \det h_{n-1} \rvert^{s+w-1} d h_{n-1} d^{\times} u \\
&= \int_{F^{\times}\bs\BA^{\times}} \int_{\GL_{n-1}(F)\bs \GL_{n-1}(\BA_F)} \Phi \left( \begin{pmatrix} z_{n-1}(u)h_{n-1} && \\ & u & \\ && 1 \end{pmatrix} \right) \varphi(h_{n-1}) \\ 
& \cdot \vert u \rvert^{n(s-\frac{1}{2})} \vert \det h_{n-1} \rvert^{s+w-1} d h_{n-1} d^{\times} u.
\end{aligned}
\end{equation}

Since $\Phi$ is a cusp form, our integral is well-defined for all complex parameters $s$ and $w$. Now, since $\Phi$ is left invariant by $Z_{n+1}(\BA)\GL_{n+1}(F)$, we see that for every $u,\;h_{n-1}$, we have (Note that $w_{12} \cdot w_{12}= I_{n+1}$)

\begin{equation}
\begin{aligned}
& \Phi \left( \begin{pmatrix} z_{n-1}(u)h_{n-1} && \\ & u & \\ && 1 \end{pmatrix} \right) \\
=& \Phi \left( \begin{pmatrix} z_{n-1}(u) && \\ & u & \\ && u \end{pmatrix} \begin{pmatrix} h_{n-1} && \\ & 1 & \\ && u^{-1} \end{pmatrix} \right) \\
=& \Phi \left( \begin{pmatrix} z_{n-1}(u) && \\ & u & \\ && u \end{pmatrix} w_{12} \begin{pmatrix} h_{n-1} && \\ & u^{-1} & \\ && 1 \end{pmatrix} w_{12} \right) \\
=& \Phi \left( \begin{pmatrix} h_{n-1} && \\ & u^{-1} & \\ && 1 \end{pmatrix} w_{12} \right) \\
=& \check{\Phi} \left( \begin{pmatrix} h_{n-1} && \\ & u^{-1} & \\ && 1 \end{pmatrix}  \right).
\end{aligned}
\end{equation}

This gives that

\begin{equation}
\begin{aligned}
I(w,\CA_s\Phi,\varphi) &= \int_{F^{\times}\bs\BA^{\times}} \int_{\GL_{n-1}(F)\bs \GL_{n-1}(\BA_F)} \Phi \left( \begin{pmatrix} z_{n-1}(u)h_{n-1} && \\ & u & \\ && 1 \end{pmatrix} \right) \varphi(h_{n-1}) \\
& \cdot \vert u \rvert^{n(s-\frac{1}{2})} \vert \det h_{n-1} \rvert^{s+w-1} d h_{n-1} d^{\times} u \\
&= \int_{F^{\times}\bs\BA^{\times}} \int_{\GL_{n-1}(F)\bs \GL_{n-1}(\BA_F)} \check{\Phi} \left( \begin{pmatrix} h_{n-1} && \\ & u^{-1} & \\ && 1 \end{pmatrix} \right) \varphi(h_{n-1}) \\ 
& \cdot \vert u \rvert^{n(s-\frac{1}{2})} \vert \det h_{n-1} \rvert^{s+w-1} d h_{n-1} d^{\times} u.
\end{aligned}
\end{equation}

Applying this to \eqref{As} and using the change of variables which is given by $(u,h_{n-1})= (u'^{-1}, z_{n-1}(u')h_{n-1}')$. We will see that the following equations hold:
$$ n(s'-\frac{1}{2})=n(\frac{1}{2}-s)+(n-1)(s+w-1)$$
and
$$ s'+w'-1=s+w-1,$$
which gives the result by solving two linear equations.

\end{proof}

\subsection{Spectral Decomposition and $L$-functions: Pre-spectral Reciprocity Formula}

Now we will give a spectral decomposition of the period $I(w,\CA_s \Phi,\varphi)$ which we consider in Proposition \ref{abstract-reciprocity-formula}.

Let $\Pi$ be an automorphic cuspidal representation for $\GL_{n+1}(\BA_F)$ with trivial central character and let $\Phi=\otimes_v'\Phi_v\in \Pi$ be a cusp form. Let $\pi_1$ be an automorphic (everywhere unramified) representation for $\GL_{n-1}(\BA_F)$ with trivial central character and let $\varphi=\otimes_v'\varphi_v\in \pi_1$ be an automorphic form.  We note that since $\Phi$ is a cusp form, therefore $\CA_s\Phi$ is of rapid decay. Hence, we can apply the abstract spectral decomposition formula in Proposition \ref{spectral theorem} as follows:

\begin{equation} \label{spectral1}
\begin{aligned}
\CA_s\Phi(h) = & \sum_{\pi\in C(S)}\sum_{\phi\in\CB_{\text{cusp}} \left(\pi\right)}\langle \CA_s\Phi, \phi \rangle \phi(h) + \sum_{\pi\in R(S)}\sum_{\phi\in\CB_{\text{res}} \left(\pi\right)}\langle \CA_s\Phi, \phi \rangle \phi(h) \\
+& \int_{\pi\in E(S)}\sum_{\phi\in\CB_{\text{cont}} \left(\pi\right)}\langle \CA_s\Phi, \phi \rangle \phi(h) d \mu_{\text{aut}}(\pi).
\end{aligned}
\end{equation}

Here $S$ is any finite set of places that contain all the archimedean places and those finite places for which $\pi_v$ is ramified. Moreover, we let $C(S)$ be the collection of cuspidal automorphic representations of $\GL(n)$ which are unramified everywhere outside $S$. We let $E(S)$ be the collection of Eisenstein series of $\GL(n)$ which are unramified everywhere outside $S$. And $R(S)$ is the collection of residual automorphic representations of $\GL(n)$ that are unramified everywhere outside $S$.

Now we take the integration on both sides of above equation \eqref{spectral1} against a standard additive character $\psi_1$ which is defined in Section 2.3 for Whittaker models and over the compact set $N_n(F)\bs N_n(\BA)$, we get the following equation for Whittaker functions:

$$
W_{\CA_s \Phi}(h)= \sum_{\pi\in C(S)}\sum_{\phi\in\CB_{\text{cusp}} \left(\pi\right)}\langle \CA_s\Phi, \phi \rangle W_{\phi}(h)+ \int_{\pi\in E(S)}\sum_{\phi\in\CB_{\text{eisen}} \left(\pi\right)}\langle \CA_s\Phi, \phi \rangle W_{\phi}(h) d \mu_{\text{aut}}(\pi).
$$

We note that since the residue spectrum for $\GL(n)$ are not generic (\cite{JL13} Proposition 2.1), they do not contribute to the above expression and therefore vanish. We recall that not all Eisenstein series from continuous spectrum are generic (see Section 2.5). They are generic if and only if they are induced from the cuspidal data. For example, in the case of $\GL(3)$, it only contains two parts, which are minimal and maximal cuspidal Eisenstein series with the corresponding partition $3=1+1+1$ and $3=2+1$. Now, if we let $\Re(w)$ be large enough, we see that

\begin{equation}
\begin{aligned}
\Psi(w, W_{\CA_s \Phi},W'_{\varphi}) &= \sum_{\pi\in C(S)}\sum_{\phi\in\CB_{\text{cusp}} \left(\pi\right)}\langle \CA_s\Phi, \phi \rangle \Psi(w, W_{\phi},W'_{\varphi}) \\
&+  \int_{\pi\in E(S)}\sum_{\phi\in\CB_{\text{eisen}} \left(\pi\right)}\langle \CA_s\Phi, \phi \rangle \Psi(w, W_{\phi},W'_{\varphi})d \mu_{\text{aut}}(\pi).
\end{aligned}
\end{equation}
Since $\CA_s\Phi$ is a rapid-decay function for $\GL(n)$, we see that

\begin{equation}
\begin{aligned}
\Psi(w, W_{\CA_s \Phi},W_{\varphi}')& = \int_{N_{n-1}(\BA_F) \bs \GL_{n-1}(\BA_F)}W_{\CA_s\Phi}\left(\begin{pmatrix} h_{n-1} & \\ & 1 \end{pmatrix} \right) W_{\varphi}'(h_{n-1}) |\det h_{n-1}|^{w-\frac12} d h_{n-1} \\
&=I\left(w,\CA_s\Phi, \varphi \right)
\end{aligned}
\end{equation}
by Proposition \ref{IY} (Note that the period integral $I\left(w,\CA_s\Phi, \varphi \right)$ is absolute convergent since $\CA_s \Phi$ is of rapid decay). We note that in fact the terms related to inner product and zeta integral $\langle \CA_s \Phi, \phi \rangle$ and $\Psi(w,W_{\phi},W_{\varphi}')$ are a product of local (zeta) integrals when $\Re(w)$ is large enough since
$$
\Psi\left(w,W_{\phi},W_{\varphi}' \right)=\prod_v\Psi_v(w,W_{\phi_v},W_{\varphi_v}').
$$
Moreover, by changing variables, we note that

\begin{equation}
\begin{aligned}
\langle \CA_s\Phi,\phi \rangle &=\int_{\RX}|\det h|^{s-\frac12}\left(\int_{F^{\times}\bs \BA^{\times}}  \Phi\begin{pmatrix} z_n(u)h_n &\\&1 \end{pmatrix} |u|^{n(s-\frac{1}{2})} du \right)\overline{\phi(h_n)} dh_n\\
&=\int_{\GL_n(F)\bs \GL_n(\BA_F)}  \Phi\begin{pmatrix} h_n &\\&1 \end{pmatrix} \overline{\phi(h_n)}|\det h_n|^{s-\frac12} dh_n=I(s,\Phi,\overline{\phi}),
\end{aligned}
\end{equation}
since $\phi$ is invariant under the center $\RZ_n(\BA)$.

By the Rankin-Selberg theory for $\GL(n+1) \times \GL(n)$, we can also write $I(s,\Phi,\overline{\phi})$ as a product of local zeta integrals (see Section 3.2). For large enough $\Re(s)$, we have
$$ I(s,\Phi,\overline{\phi})= \Psi(s,W_{\Phi},\overline{W_{\phi}})=\Psi(s,W_{\Phi},W_{\overline{\phi}}')= \prod_v \Psi_v (s,W_{\Psi_v},W_{\overline{\phi_v}}')=\prod_v \Psi_v (s,W_{\Psi_v},\overline{W_{\phi_v}}).$$

Combining above discussions together, we have the following Proposition:
\begin{prop}\label{spectral-term}
Let $\Pi$ be a cuspidal automorphic representation with trivial central character and let $\Phi=\otimes_v' \Phi_v \in \Pi$ be a cusp form for $\GL(n+1)$ over $F$. Let $\pi_1$ be an automorphic (everywhere unramified) representation for $\GL(n-1)$ over $F$ with trivial central character and let $\varphi=\otimes_v'\varphi_v\in \pi_1$ be an automorphic form. Let $\tilde{\pi}$ be the contragredient representation of $\pi$. Then, we have

\begin{equation}
\begin{aligned}
2 \Delta_F^{1/2}I\left(w,\CA_s\Phi,\varphi \right)& =\sum_{\pi \in C(S),\, \pi \, \mathrm{cusp}}\frac{\Lambda(s,\Pi\times \widetilde{\pi} )\Lambda(w,\pi \times \pi_1)}{\Lambda(1,\operatorname{Ad}\pi)}H(\pi)\\
&+ \int_{\pi \in E(S),\, \pi \, \mathrm{eisen}}\frac{\Lambda(s,\Pi\times \widetilde{\pi} )\Lambda(w,\pi \times \pi_1)}{\Lambda^{*}(1,\operatorname{Ad}\pi)}H(\pi) d \mu_{\mathrm{aut}}(\pi),
\end{aligned}
\end{equation}
where $H(\pi)=\prod_v H_v(\pi_v)$ is the global weight function. We note that $H_v$ depends on the choice of $\Phi_v$ and $\varphi_v$, $s$ and $w$, which is given by

\begin{equation}\label{Hv-def}
H_v(\pi_v):=\sum_{W \in \CB^W(\pi_v)}\frac{\Psi_v(s,W_{\Phi_v},\overline{W})\Psi_v(w,W,W_{\varphi_v}')}{L_v(s,\Pi_v\times \widetilde{\pi}_v)L_v(w,\pi_v \times \pi_{1,v})}.
\end{equation}

Here for each (generic) automorphic representation $\pi$ of $\GL(n)$, we will consider the following completed $L$-functions
$$
\Lambda(s,\operatorname{Ad},\pi),\, \Lambda(s,\Pi\times\pi) \text{ and }\Lambda(s,\pi\times\pi_1).
$$
They are, the Adjoint $L$-function of $\pi$, the Rankin-Selberg $L$-function of $\Pi \times \pi$ and $\pi \times \pi_1$, respectively.

Moreover, $\Lambda^{*}(1,\operatorname{Ad}\pi)$ in the denominator means the non-zero residue of the completed adjoint $L$-functions. If $\pi$ is cuspdial, we know that $\Lambda^{*}(1,\operatorname{Ad}\pi)=\Lambda(1,\operatorname{Ad}\pi) \neq 0$.
\end{prop}

We actually have the following:

Let $F$ be a number field, with ring of integers $\CO_F$. Let $\Pi$ be a cuspidal automorphic representation of $\GL(n+1)$ over $F$ with trivial central character. Let $\pi_1$ be an automorphic representation of $\GL(n-1)$ over $F$ with trivial central character.

 Let $H$ be a global weight function which is defined above. We consider the following sums: 

$$
\CC(H):=\sum_{\pi\in C(S)}\frac{\Lambda(s,\Pi\times \widetilde{\pi})\Lambda(w,\pi \times \pi_1)}{\Lambda(1,\operatorname{Ad},\pi)}H(\pi),
$$
which is the cuspidal contribution.

We should also consider the following continuous (Eisenstein) contribution
\begin{equation} \label{E-def}
\CE(H):=\int_{\pi\in E(S)}\frac{\Lambda(s,\Pi\times \widetilde{\pi})\Lambda(w,\pi \times \pi_1)}{\Lambda^{*}(1,\operatorname{Ad},\pi)}H(\pi) d \mu_{\text{aut}}(\pi),
\end{equation}

where $S$ is any finite set of places that contain all the archimedean places and those finite places for which $\pi_v$ is ramified. Moreover, we let $C(S)$ be the collection of cuspidal automorphic representations of $\GL(n)$ which are unramified everywhere outside $S$. We let $E(S)$ be the collection of Eisenstein series of $\GL(n)$ which are unramified everywhere outside $S$.

By definition, we may write $H(\pi)=H(s,w,\pi,\Phi)$. We define $\check{H}(\pi):=H(s',w',\pi,\check{\Phi})$, where $s,w'$ and $\check{\Phi}$ is defined in the proof of abstract reciprocity formula (Proposition \ref{abstract-reciprocity-formula}).

We want to understand the following first moment of generic spectrum $\hat{\RX}_{\text{gen}}$ which is the summation of the cuspidal contribution and Eisenstein contribution:
$$
\CM(H):=\CC(H)+\CE(H).
$$

We end this subsection with the following result which can be seen as a pre-spectral reciprocity formula.

\begin{thm}\label{pre}
Let $s,w\in\BC$ and define

\begin{equation}
(s',w'):=\left(\frac{1+(n-1)w-s}{n}, \frac{(n+1)s+w-1}{n} \right).
\end{equation}

Let $S$ be a finite set of places which contain all the archimedean places and those finite places for which $\pi_v$ is ramified. Suppose that the real parts of four parameters $s,w,s',w'$ are all  sufficiently large. Then we have the following identity

$$
\CM(H)=\CM(\check{H}).
$$ 

\end{thm}

\begin{proof}
This is a direct corollary from Proposition \ref{abstract-reciprocity-formula} and Proposition \ref{spectral-term}.
\end{proof}

\subsection{Local Vectors and Computations}

Now in order to give the explicit spectral reciprocity formula, we have to pick local vectors. We follow the method in \cite[Section 7]{Nun20}. We will show how to choose local vectors on some special non-archmediean places and archmediean places. 
We also give some very basic local estimations on these places for local weight functions $H_v$. Our local estmations are incomplete but are enough for us to establish the spectral reciprocity formula.

For simplicity, we further assume that $\pi_1$ is an automorphic cuspidal (everywhere unramified) representation for $\GL(n-1)$ with trivial central character from this subsection. Let $\varphi=\otimes_v'\varphi_v\in \pi_1$ be a cusp form. Moreover, for every place $v$, we simply fix $\varphi_v:=\varphi_v^0$ be the normalized spherical vector in the Whittaker model.

Let $\Pi$ be an unramified (everywhere) cuspidal automorphic representation of $\GL(n+1)$ with trivial central character. For all the place $v$, we let $\Phi_v^{0}$ be the normalized spherical vector in the Whittaker model. Let $\Fq$ and $\Fl$ be two fixed (may not unramified) coprime integral ideals of $F$. We will write $\Phi^{\Fq,\Fl}=\otimes'_v \Phi^{\Fq,\Fl}_v \in \Pi$ be a cusp form. For all $v\nmid \Fq \Fl$, we will simply pick $\Phi^{\Fq,\Fl}_v=\Phi^{0}_v$. For $v\mid \Fl$, we will pick
\begin{equation}\label{Phi-l}
\Phi_v^{\Fq,\Fl}(g):=\frac{1}{p_v^{(n-1)k}}\sum_{\beta_i \,\in\, \Fm_v^{-k}/\Fo_v,\,i=1,2,\cdots,n-1} \Phi^0_v\left(g\begin{pmatrix} 1&&&& \beta_1 &\\&1&&& \beta_2 &\\&& \cdots &&& \\ &&& 1 & \beta_{n-1} & \\ &&&& 1 & \\ &&&&& 1 \end{pmatrix} \right),
\end{equation}
where $k=v(\Fl)$. Finally, for $v\mid \Fq$, we will pick
\begin{equation}\label{Phi-q}
\Phi_v^{\Fq,\Fl}(g):=\frac{1}{p_v^{(n-1)m}}\sum_{\beta_i \,\in\, \Fm_v^{-m}/\Fo_v,\,i=1,2,\cdots,n-1} \Phi^0_v\left(g\begin{pmatrix} 1&&&&& \beta_1 \\&1&&&& \beta_2 \\&& \cdots &&& \\ &&& 1 && \beta_{n-1}  \\ &&&& 1 & 0 \\ &&&&& 1  \end{pmatrix} \right),
\end{equation}
with $m=v(\Fq)$.

The choice for the local vector here for $v \vert \Fq$ is compatible with the local new-vector computation made in \cite{BKL19}.

\begin{rmk}
In order to design a spectral reciprocity formula, we do not have much freedom when choosing the local vectors. For the reciprocal relation of two unramified coprime ideals $\Fq$ and $\Fl$, we only have the freedom for one finite place. For example, after we pick the local vector for the place $v \vert \Fq$, the local vector for the place $v \vert \Fl$ is automatically fixed. They are related by the following simple matrix identity:
$$ \begin{pmatrix} 1&&&&& \beta_1 \\&1&&&& \beta_2 \\&& \cdots &&& \\ &&& 1 && \beta_{n-1}  \\ &&&& 1 & 0 \\ &&&&& 1  \end{pmatrix} = \begin{pmatrix} I_{n-1} && \\ &&1 \\ &1& \end{pmatrix} \begin{pmatrix} 1&&&& \beta_1 &\\&1&&& \beta_2 &\\&& \cdots &&& \\ &&& 1 & \beta_{n-1} & \\ &&&& 1 & \\ &&&&& 1 \end{pmatrix} \begin{pmatrix} I_{n-1} && \\ &&1 \\ &1& \end{pmatrix}.$$
\end{rmk}

We let $S$ be a finite set with the definition $S=\{v\mid\Fq\}\cup\{v\mid\Fl\}\cup\{v\mid\infty\}$. We have the following local properties for several different cases (The computation will be given later).
\begin{itemize}
\item If $v \notin S$, we know that $\pi_v$ is unramified and $\Phi_v^{\Fq,\Fl}=\Phi_v^0$ is spherical. And we also have $H_v(\pi_v)=1$ in this case with only one term survives in the summation of $H_v$. This is a direct corolloary from the discussion in Section 3.2.

\item If $v \mid \Fq$, let $\Phi_v=\Phi_v^{\Fq,\Fl}$ be as in \eqref{Phi-q} and let  $t=v(\Fq)$. Then we have

(1) $H_v(\pi_v)$ vanishes if $\operatorname{cond} (\pi_v) >t$.

(2) $H_v(\pi_v)=p^{-(n-1)t}$ if $\operatorname{cond} (\pi_v )= t$.

\item If $v \mid \infty$ which is the Archimedean place, we have the following Proposition
\begin{prop}\label{Archimedean calculation}
Let $v$ be an Archimedean place of $F$. Let $\Pi_v$ be an irreducible admissible generic representation for $\GL_{n+1}(F_v)$, then there exists  a Whittaker function $W_{\Pi_v}\in \CW(\Pi_v,\psi_v)$ such that for every irreducible admissible generic representation for $\GL_n(F_v)$, we have

$$
H_v(\pi_v)=
\begin{cases}
1,\text{ if }\pi_v\text{ is unramified},\\
0,\text{ otherwise}.
\end{cases}
$$
\end{prop}

\begin{proof}
This Proposition is proved in \cite{Nun20} Section 7.4 by applying Stade's formula (Theorem 3.4 in \cite{Sta01}) for Archimedean $\GL(n+1) \times \GL(n)$ Rankin-Selberg $L$-functions.

\end{proof}

\end{itemize}

\textbf{Some local computations for two fixed finite places}

For $v\mid \Fl$, we pick
\begin{equation}
\Phi_v^{\Fq,\Fl}(g):=\frac{1}{p_v^{(n-1)k}}\sum_{\beta_i \,\in\, \Fm_v^{-k}/\Fo_v,\,i=1,2,\cdots,n-1} \Phi^0_v\left(g\begin{pmatrix} 1&&&& \beta_1 &\\&1&&& \beta_2 &\\&& \cdots &&& \\ &&& 1 & \beta_{n-1} & \\ &&&& 1 & \\ &&&&& 1 \end{pmatrix} \right),
\end{equation}
where $k=v(\Fl)$.

For $v\mid \Fq$, we will pick
\begin{equation}
\Phi_v^{\Fq,\Fl}(g):=\frac{1}{p_v^{(n-1)m}}\sum_{\beta_i \,\in\, \Fm_v^{-m}/\Fo_v,\,i=1,2,\cdots,n-1} \Phi^0_v\left(g\begin{pmatrix} 1&&&&& \beta_1 \\&1&&&& \beta_2 \\&& \cdots &&& \\ &&& 1 && \beta_{n-1}  \\ &&&& 1 & 0 \\ &&&&& 1  \end{pmatrix} \right),
\end{equation}
with $m=v(\Fq)$.

For $v \mid \Fl$, 
we recall that

$$
H_v(\pi_v)=\sum_{W \in \CB^W(\pi_v)}\frac{\Psi_v(s,W_{\Phi_v^{\Fq,\Fl}},\overline{W})\Psi_v(w,W,W_{\varphi_v}')}{L_v(s,\Pi_v\times \widetilde{\pi}_v)L_v(w,\pi_v \times \pi_{1,v})},
$$

where
$$
\Phi_v^{\Fq,\Fl}(g):=\frac{1}{p_v^{(n-1)k}}\sum_{\beta_i \,\in\, \Fm_v^{-k}/\Fo_v,\,i=1,2,\cdots,n-1} \Phi^0_v\left(g\begin{pmatrix} 1&&&& \beta_1 &\\&1&&& \beta_2 &\\&& \cdots &&& \\ &&& 1 & \beta_{n-1} & \\ &&&& 1 & \\ &&&&& 1 \end{pmatrix} \right),
$$
with 
$k=v(\Fl)$.

Following the same method line by line in \cite[Section 7]{Nun20}, we will see that $H_v$ vanishes unless $\pi_v$ is unramified. By right $\GL_n(F_v)$-invariance of the Haar measure, for any fixed element $h\in \GL_n(F_v)$, we see that
$$
\Psi\left(s,\Pi_v \begin{pmatrix} h&\\ & 1 \end{pmatrix} W_{\Phi_v}
,\overline{W}\right)=\Psi\left(s,W_{\Phi_v}
,\overline{\pi_v(h)W}\right).
$$

Now we note that given a basis $\CB^W(\pi_v)$ of $\CW(\pi_v,\psi_v)$, we may create a different one by considering the set
$$
\left\{\pi_v(h)\cdot W,\,W \in \CB^W(\pi_v)\right\},
$$
for some fixed element $h\in \GL_n(F_v)$. Applying this idea to the element $h=\begin{pmatrix} 1&&&& \beta_1 \\&1 &&& \beta_2 \\ && \cdots && \\ &&& 1 & \beta_{n-1} \\ &&&& 1 \end{pmatrix}$ for $\beta_i \in \Fm_v^{-k}/\Fo_v,\,i=1,2,\cdots,n-1$, we deduce that
$$	
H_v(\pi_v)=\sum_{W\in \CB^W(\pi_v)}\frac{\Psi_v(s,W_{\Phi_v^0},\overline{W})\Psi_v(w,W^{(m)},W_{\varphi_v}')}{L_v(s,\Pi_v\times \widetilde{\pi}_v)L_v(w,\pi_v \times \pi_{1,v})},
$$
where
\begin{equation}
\begin{aligned}
W^{(m)}(h'):&=\sum_{\beta\in \Fm_v^{-k}/\Fo_v}W\left(h' \begin{pmatrix} 1&&&& \beta_1 \\&1 &&& \beta_2 \\ && \cdots && \\ &&& 1 & \beta_{n-1} \\ &&&& 1 \end{pmatrix}^{-1} \right)\\
&= \sum_{\beta\in \Fm_v^{-k}/\Fo_v}W\left(h' \begin{pmatrix} 1&&&& - \beta_1 \\&1 &&& - \beta_2 \\ && \cdots && \\ &&& 1 & - \beta_{n-1} \\ &&&& 1 \end{pmatrix} \right).
\end{aligned}
\end{equation}
Since $W \mapsto \overline{\Psi_v(s,W_{\Phi_v^0},\overline{W})}$ is a right $K_v$ invariant linear functional, it vanishes if $\pi_v$ is not unramified. This kind of idea will also be used in the calculation of another local vector when $v \vert \Fq$. Moreover, this linear functional is invariant by orthogonal projection into the space $\CW(\pi_v,\psi_v)^{K_v}$ of right $K_v$-invariant vectors of $\CW(\pi_v,\psi_v)$. Since $K_v$ is the maximal open compact subgroup, we know that the space $\CW(\pi_v,\psi_v)^{K_v}$ is a one-dimensional space that is spanned by the normalized spherical vector. Therefore, we may restrict the sum defining $H_v$ to a sum over a basis of $\CW(\pi_v,\psi_v)^{K_v}$. Hence, only one term survives in the sum defining $H_v$ over the basis. Now, by the unramified calculation, we have $\Psi(s,W_{\Phi_v^0},\overline{W_{\pi_v}})=L_v(s,\Pi_v\times \widetilde{\pi}_v)$ 

By above discussion, we see that
$$ H_v(\pi_v)= \frac{\Psi_v(w,W^{(m)},W_{\varphi_v}')}{L_v(w,\pi_v \times \pi_{1,v})}$$
if $\pi_v$ is unramified.

We will write $H_v(\pi_v)$ explicitly.
From the definition of $W^{(m)}(h)$, we see that
\begin{equation}
\begin{aligned}
W^{(m)} \begin{pmatrix} h_{n-1} & \\ & 1 \end{pmatrix} &= \sum_{\beta_i \,\in\,\Fm_v^{-m}/\Fo_v,i=1,2,\cdots,n-1}\psi \left (- \sum_{i=1}^{n-1} \beta_i h_{n-1,i} \right) \cdot W_{\pi_v}\begin{pmatrix} h_{n-1} &  \\ & 1 \end{pmatrix}\\
&= p_v^{(n-2)m} \times \delta_{v(h_{n,i})\geq m,\,m=1,2,\cdots,n-1} \cdot W_{\pi_v} \begin{pmatrix} h_{n-1} & \\ & 1 \end{pmatrix}.
\end{aligned}
\end{equation}

We need to continuous our local computation by applying Iwasawa decomposition to $h_{n-1}$. We write $h_{n-1}=z(h)n(h)a(h)k(h)$. Since the valuation $v(h_{n-1,i})\geq m$ for all $i=1,2,\cdots,n-1$, we see that $v(zk_{n-1,i}) \geq m$ for all $i=1,2,\cdots,n-1$.

From above discussion, we may have the decomposition:
\begin{equation}
\Psi_v(w,W^{(m)},W_{\varphi_v}')= \sum_{\nu=m}^{\infty} p_{v}^{-(n-1) \nu (s-\frac12)} \cdot p_v^{(n-2)m} \cdot \Psi_{\nu}(W_{\pi_v}),
\end{equation}
where
\begin{equation}
\begin{aligned}
\Psi_{\nu}(W_{\pi_v}) &=\int_{A_{n-1}(F_v^{\times})}\int_{K_v}W_{\pi_v}\left(z(\varpi_v^{\nu})a \right) W'_{\varphi_v} (ak) \cdot |a|^{s-\frac{n}{2}} d A_{n-1}(F_v^{\times})\, dk \\
&= \int_{A_{n-1}(F_v^{\times})} W_{\pi_v}\left(z(\varpi_v^{\nu})a \right)  W_{\varphi_v}(a) \cdot |a|^{s-\frac{n}{2}} d A_{n-1}(F_v^{\times}).
\end{aligned}
\end{equation}
Here we use the fact that both the Whittaker function $W_{\pi_v}$ and $W_{\varphi_v}'$ are unramified.  Hence they are invariant under the maximal compact subgroup $K_v$ and we also note that the total mass for $K_v$ is one. Here the local weight function $H_v$ should be related to the local Hecke eigenvalue $\lambda_{\pi_v}$. \\

For $v \mid \Fq$, 
we recall that

$$
H_v(\pi_v)=\sum_{W \in \CB^W(\pi_v)}\frac{\Psi_v(s,W_{\Phi_v^{\Fq,\Fl}},\overline{W})\Psi_v(w,W,W_{\varphi_v}')}{L_v(s,\Pi_v\times \widetilde{\pi}_v)L_v(w,\pi_v \times \pi_{1,v})},
$$

where
$$
\Phi_v^{\Fq,\Fl}(g):=\frac{1}{p_v^{(n-1)m}}\sum_{\beta_i \,\in\, \Fm_v^{-m}/\Fo_v,\,i=1,2,\cdots,n-1} \Phi^0_v\left(g\begin{pmatrix} 1&&&&& \beta_1 \\&1&&&& \beta_2 \\&& \cdots &&& \\ &&& 1 && \beta_{n-1}  \\ &&&& 1 & 0 \\ &&&&& 1  \end{pmatrix} \right),
$$
with $m=v(\Fq)$.
This means that
$$
W_{\Phi_v^{\Fq,\Fl}}(g):= \frac{1}{p_v^{(n-1)m}}\sum_{\beta_i \,\in\, \Fm_v^{-m}/\Fo_v,\,i=1,2,\cdots,n-1} W_{\Phi_v^0} \left(g\begin{pmatrix} 1&&&&& \beta_1 \\&1&&&& \beta_2 \\&& \cdots &&& \\ &&& 1 && \beta_{n-1}  \\ &&&& 1 & 0 \\ &&&&& 1  \end{pmatrix} \right),
$$
with $m=v(\Fq)$.
Here $W_{\Phi_v^0}$ is the normalized spherical vector in the Whittaker space. By the theory of newvectors \cite{JPS81}, we will see that $H_v(\pi_v)$ vanishes unless $\text{cond}(\pi_v) \leq m$. If $\text{cond}(\pi_v)=m$, $H_v(\pi_v)=p_v^{-(n-1)m}$. We expect that in general, $H_v(\pi_v) \ll_{\epsilon} p_v^{(n-1)m(\theta-1+\epsilon)}$, where $\theta$ is a positive constant satisfying $0 \leq \theta< \frac{1}{2}$ towards the Ramanujan-Petersson's Conjecture.

The computation in this part follows the method developed in \cite{BKL19} and \cite[Section 7]{Nun20}.

By definition and some matrix computations, we note that
\begin{equation}
\begin{aligned}
W_{\Phi_v^{\Fq,\Fl}}\begin{pmatrix} h_n & \\ & 1 \end{pmatrix} &=\frac{1}{p_v^{(n-1)m}} \times \sum_{\beta_i \,\in\,\Fm_v^{-m}/\Fo_v,i=1,2,\cdots,n-1}\psi \left (\sum_{i=1}^{n-1} \beta_i h_{n,i} \right) \cdot W_{\Phi_v^0}\begin{pmatrix} h_n &  \\ & 1 \end{pmatrix}\\
&=\delta_{v(h_{n,i})\geq m,\,m=1,2,\cdots,n-1} \cdot W_{\Phi_v^0}\begin{pmatrix} h_n & \\ & 1 \end{pmatrix}.
\end{aligned}
\end{equation}

We need to continuous our local computation by applying Iwasawa decomposition to $h_n$. We write $h_n=z(h)n(h)a(h)k(h)$. Since the valuation $v(h_{n,i})\geq m$ for all $i=1,2,\cdots,n-1$, we see that $v(zk_{n,i}) \geq m$ for all $i=1,2,\cdots,n-1$. We set $a_1:=\min(m,v(z))$ and $a_2:=m-a_1$, we can see that this is equivalent to say that $k$ belonging to the congruence subgroup $K_{v,0}(\varpi_v^{a_2})$ (See the definition in Section 2.2). We will simply write this congruence subgroup as $K_v[a_2]$. Here we must have $v(z) \geq 0$. Otherwise, the spherical Whittaker function will vanish. Therefore we have $0 \leq a_1,a_2 \leq m$.

Now we may choose an orthonormal basis for $\CW(\pi_v,\psi_v)$. From the local new-vector theory for $\GL(n)$ \cite{JPS81}, if we let $W_0=W_{\pi_v}$ be the new vector and for each $j \geq 0$, we define
$$ W_j:= \pi_v \begin{pmatrix} 1&&& \\ &1&& \\ && \cdots & \\ &&& \varpi_v^j \end{pmatrix} W_0. $$
Therefore, we see that $\left\{W_0,\,W_1,W_2,\ldots\right\}$ is a basis for $\CW(\pi_v,\psi_v)$. Moreover, we know that for each $j\geq 0$, $\left\{W_0,\,W_1,\ldots, W_j\right\}$ is a basis for the $K_v[n_0+j]$-invariant vectors in $\CW(\pi_v,\psi_v)$, where $n_0=\operatorname{cond}(\pi_v)$. Applying Gram-Schmidt method to the basis $\left\{W_0,\,W_1,W_2,\ldots\right\}$, we obtain an orthonormal basis of $\CW(\pi_v,\psi_v)$ with $\{\widetilde{W_0},\widetilde{W_1},\widetilde{W_2},\ldots\}$. We choose the basis as $\CB^W(\pi_v)=\{\widetilde{W_0},\widetilde{W_1},\widetilde{W_2},\ldots\}$ and continue to do some computation on $H_v(\pi_v)$.\\

From above discussion, we have the following decomposition:

\begin{equation} \label{Phi1}
\Psi_v(s,W_{\Phi_v^{\Fq,\Fl}},\overline{W})=\sum_{a_1+a_2=m}\sum_{\min(\nu,m)=a_1}p_v^{-n \nu(s-\frac12)}\Psi_{\nu,a_2}(W),
\end{equation}
where
\begin{equation}
\begin{aligned}
\Psi_{\nu,a_2}(W) &=\int_{A_n(F_v^{\times})}\int_{K_v[a_2]}W_{\Phi_v^0}\left(z(\varpi_v^{\nu})a \right)\overline{W}(ak) \cdot |a|^{s-\frac{n+1}{2}} d A_n(F_v^{\times})\, dk \\
&= \int_{A_n(F_v^{\times})} W_{\Phi_v^0}\left(z(\varpi_v^{\nu})a \right) \int_{K_v[a_2]} \overline{W}(a k) \cdot |a|^{s-\frac{n+1}{2}} d A_n(F_v^{\times}) \, dk.
\end{aligned}
\end{equation}
Here we use the fact that Whiitaker function $W_{\Phi_v^0}$ is normalized spherical and $W$ is also invariant by the center. Now, if $W=\widetilde{W_j}$ is an element of in our orthonormal basis $\{\widetilde{W_0},\widetilde{W_1},\widetilde{W_2},\ldots\}$, then it follows that

\begin{equation}\label{compact1}
\int_{K_v[f]}\widetilde{W_j}(hk) dk=
\begin{cases}
\operatorname{vol}(K_v[f])\widetilde{W_j}(h),\text{ if }j+n_0\leq f,\\
0,\text{ otherwise}.
\end{cases}
\end{equation}
from the orthogonality of the elements in the given orthonormal basis (See the discussion in \cite{Nun20} for more details).

Now applying \eqref{Phi1} and \eqref{compact1} to the definition of the local weight function $H_v(\pi_v)$. We will have the following equation:

\begin{equation} 
\begin{aligned}
H_v(\pi_v)&=\frac{1}{L_v(s,\Pi_v\times \widetilde{\pi}_v)L_v(w,\pi_v \times \pi_{1,v})}\sum_{a_1+a_2=m}\sum_{j\leq a_2-n_0}\operatorname{vol}(K_v[j])\sum_{\min(\nu,m)=a_1}p^{-n \nu(s-1/2)} \\
& \times \int_{A_n(F_v^{\times})}W_{\Phi_v^0}\begin{pmatrix} z(\varpi_v^{\nu})a &\\& 1 \end{pmatrix}\overline{\widetilde{W}_j}(a))|a|^{s-(n+1)/2} d A_n(F_v^{\times}) \Psi_v(w,\widetilde{W_j},W_{\varphi_v}').
\end{aligned}
\end{equation}
Note that the above equation is actually a finite sum. If $\text{cond}(\pi_v)>m$, this gives that $n_0>m$. We see that $0 \leq j \leq a_2-n_0 \leq m-n_0<0$. Contradiction! Therefore, $H_v(\pi_v)$ vanishes in this case.

If $\text{cond}(\pi_v)=m$, we see that only one term survive. Moreover, we have $a_2=m$, and $j=a_1=0$. We see that $H_v(\pi_v)= \text{vol}(K_v[a_2])^{-1}=\text{vol}(K_v[m])^{-1}=p_v^{-(n-1)m}$ in this case.

Note that since $0 \leq j \leq a_2-n_0 \leq a_2 \leq m$, the local weight function $H_v$ is a finite sum in terms of the element $\widetilde{W}_j$ in the orthonormal basis.

For two places $\Fl$ and $\Fq$, we see that the local weight function $H_v$ is a finite sum with the elements in the orthonormal basis $\CW(\pi_v,\psi_v)$ for both two cases.

\begin{rmk}
In order to find applications for the spectral reciprocity formula, we have to find a good estimation for the local weight function $H_v(\pi_v)$.
\end{rmk}

\subsection{Meromorphic Continuation with respect to the complex parameters}


Let $v$ be a non-archimedean place of $F$. We consider the local weight function $H_v(\pi_v)$. It is known that for all the finite place $v \neq \Fq, \Fl$, we have $H_v(\pi_v)=1$.

Assume that $\pi_v$ is a local component of the unitary generic Eisenstein series $\pi=\pi(\pi_1,\pi_2,\cdots,\pi_k,t_1,\cdots,t_k)=\pi_1 \vert \cdot \rvert^{i t_1} \boxplus\dots\boxplus\pi_k \vert \cdot \rvert^{i t_k}$. We fix the cuspidal data $\pi_1,\cdots,\pi_k$ and vary the remaining parameters $(t_1,t_2,\cdots,t_k) \in \BC^k$ ($k \geq 2$).

The following Proposition is a direct Corollary of Proposition 4.1, Proposition 4.2 and Theorem 4.1 in \cite{CPS17}.

\begin{prop}
The following two ratios
$$ \frac{\Psi_v(s,W_{\Phi_v},\overline{W})}{L_v(s,\Pi_v \times \widetilde{\pi}_v)}$$
and
$$ \frac{\Psi_v(w,W,W_{\varphi_v}')}{L_v(w,\pi_v \times \pi_{1,v})}$$
for any $W \in \CW(\pi_v,\psi_v)$ have no poles and hence define entire rational functions in terms of their complex parameters. In other word, we have
$$ \frac{\Psi_v(s,W_{\Phi_v},\overline{W})}{L_v(s,\Pi_v \times \widetilde{\pi}_v)} \in \BC[p_v^{s},p_v^{-s},p_v^{t_1},p_v^{-t_1},\cdots,p_v^{t_k},p_v^{-t_k}],$$
and
$$ \frac{\Psi_v(w, W, W_{\varphi_v}')}{L_v(w,\pi_v \times \pi_{1,v})} \in \BC[p_v^{w},p_v^{-w},p_v^{t_1},p_v^{-t_1},\cdots,p_v^{t_k},p_v^{-t_k}].$$
\end{prop}

Since the local weight function $H_v$ is a finite sum with the elements in the orthonormal basis $\CW(\pi_v,\psi_v)$, we have the following Proposition.

\begin{prop} \label{holomorphic}
The local weight function $H_v(\pi_v)$ has no poles and hence define entire rational functions in terms of their complex parameters. In other word, we have
$$ H_v(\pi_v) \in \BC[p_v^{s},p_v^{-s},p_v^w, p_v^{-w}, p_v^{t_1},p_v^{-t_1},\cdots,p_v^{t_k},p_v^{-t_k}].$$
\end{prop}

From our choice of local vectors in the previous subsection and the definition of global and local weight function, we know that
$$ H(\pi)= \prod_v H_v(\pi_v)= H_{\Fq}(\pi_{\Fq})\times H_{\Fl}(\pi_{\Fl}),$$
since $H_v(\pi_v)=1$ if $v\nmid \Fq \Fl$.

Therefore, we have the following proposition:

\begin{prop} \label{holomorphic1}
The global weight function $H(\pi)$ has no poles and hence define entire rational functions in terms of their complex parameters.

\end{prop}

We are going to deduce the meromorphic continuation of the term $\CE(H)$ which is the continuous contribution of Eisenstein series in the spectral decomposition. 

\begin{prop}

Let $\Pi$ and $\pi_1$ be everywhere unramified cuspidal automorphic representation of $\GL(n+1)$ and $\GL(n-1)$ which have trivial central character. Let $\CE(H)$ \eqref{E-def} be given previously, defined initially for the absolute convergence domain $\Re(s),\Re(w)> 1$. It admits a meromorphic continuation to $\Re(s),\Re(w)\geq \frac12$ . If $\frac12\leq \Re(s),\Re(w)<1$, its analytic continuation is given by $\CE(H)+\CR(H)$, where

\begin{equation} \label{R-def}
\begin{aligned}
\CR(H):& =\sum_{ \omega\in\widehat{F^{\times}U_{\infty}\bs \BA_F^1} }  \underset{\substack{t_1= (1-w)/i}}{\operatorname{Res}}( 2 \pi i) \times \\
& \frac{\Lambda(s-it_1,\Pi \times \pi_1)\Lambda(s-it_2,\Pi \times \omega^{-1})\Lambda(w+it_1, \widetilde{\pi}_1 \times \pi_1) \Lambda(w+it_2,\pi_1 \times \omega)}{\Lambda^*(1,\operatorname{Ad},\pi)}H(\pi).
\end{aligned}
\end{equation}

Note that the summation over the unitary Hecke character $\omega$ which is only ramified at two finite places $\Fq$ and $\Fl$ is a finite sum. We write $U_{\infty}:=\prod_{v\mid \infty}\{y\in F_v^{\ast};|y_v|=1\}$ and recall that $\BA_F^1$ is the norm one ideles. Here we pick $\pi$ to be the maximal cuspidal Eisenstein series given by $\pi=\pi(\widetilde{\pi}_1,\omega,t_1,t_2)=\widetilde{\pi}_1 \vert \cdot \rvert^{i t_1} \boxplus \omega \vert \cdot \rvert^{i t_2}$, where $\widetilde{\pi_1}$ is the contrigredient representation of $\pi_1$. Hence it is a cuspdial representation of $\GL(n-1)$ with trivial central character. We have the following decomposition of completed $L$-functions:
$$ \Lambda(s,\Pi \times \widetilde{\pi})=\Lambda(s-it_1,\Pi \times \pi_1)\Lambda(s-it_2,\Pi \times \omega^{-1}),$$
and
$$ \Lambda(w,\pi \times \pi_1)= \Lambda(w+it_1, \widetilde{\pi}_1 \times \pi_1) \Lambda(w+it_2,\pi_1 \times \omega).$$
Moreover, for general $\pi=\pi(\sigma_1,\sigma_2,t_1,t_2)$, we see that $t_1+t_2=A$, where $A$ is a complex constant only depends on the central characters of the cuspidal data $\sigma_1$ and $\sigma_2$ since $\pi$ has trivial central character.

\end{prop}

\begin{proof}
The meromorphic continuation part of $\CE(H)$ is given by the meromorphic continuation of Rankin-Selberg $L$-functions and the entireness of the global weight function $H(\pi)$ (see Proposition \ref{holomorphic1}). Since $H(\pi)$ has no poles and define entire rational functions in terms of their parameters (Proposition \ref{holomorphic1}), by the contour and residue theorem in complex analysis, we see that the term $\CR(H)$ will vanish unless the ratio of completed $L$-functions
$$ \frac{\Lambda(s,\Pi \times \widetilde{\pi})\Lambda(w,\pi \times \pi_1)}{\Lambda^*(1,\text{Ad},\pi)}$$
have poles. Note that the denominator $\Lambda^*(1,\text{Ad},\pi)$ is always finite and non-zero, by the locations of possible poles of Rankin-Selberg $L$-functions, we must have
$$\pi=\pi(\widetilde{\pi}_1,\omega,t_1,t_2)=\widetilde{\pi}_1 \vert \cdot \rvert^{i t_1} \boxplus \omega \vert \cdot \rvert^{i t_2},$$
where $\widetilde{\pi_1}$ is the contrigredient representation of $\pi_1$ and $\omega$ is a unitary Hecke character.

Hence we have the decomposition of completed $L$-functions: 
$$ \Lambda(s,\Pi \times \widetilde{\pi})=\Lambda(s-it_1,\Pi \times \pi_1)\Lambda(s-it_2,\Pi \times \omega^{-1}),$$
and
$$ \Lambda(w,\pi \times \pi_1)= \Lambda(w+it_1, \widetilde{\pi}_1 \times \pi_1) \Lambda(w+it_2,\pi_1 \times \omega).$$

We note that $\Lambda(s, \Pi \times \widetilde{\pi})$ is entire for all $s,t_1,t_2\in \BC$ since $\Pi$ is a cuspidal automorphic representation for $\GL(n+1)$. The completed $L$-function $\Lambda(w,\pi \times \pi_1)$ will have a simple pole if and only if $w+it_1=1$. The correponding residue is
$$ \Lambda^*(w, \pi \times \pi_1)= \Lambda(1,\text{Ad},\pi_1)\Lambda(w+it_2,\pi_1 \times \omega).$$
Now applying the coutour and residue theorem in complex analysis, the remaining part of the proof is the same as the proof in \cite[Proposition 8.1]{Nun20}, \cite[Lemma 16]{BK17} and \cite[Lemma 3]{BK18}.
\end{proof}

\subsection{Main Result: Spectral Reciprocity Formula}

Now we can give the statement of spectral reciprocity formula.

Let $\Pi$ and $\pi_1$ be everywhere unramified cuspidal automorphic representation for $\GL(n+1)$ and $\GL(n-1)$ over $F$ with trivial central character. Let $s, w\in \BC$, $\Fq$ and $\Fl$ be unramified coprime ideals. 
 
Let $H$ be the global weight function with kernel function $\Phi=\Phi^{\Fq,\Fl}$ which we pick in the previous subsection by local new-vectors.

Note that we have $$H(\pi)=\prod_v H_v(\pi_v),$$
where $H_v$ is given by local Rankin-Selberg integral as follows:
\begin{equation}\label{Hv-def}
H_v(\pi_v):=\sum_{W \in \CB^W(\pi_v)}\frac{\Psi_v(s,W_{\Phi_v},\overline{W})\Psi_v(w,W,W_{\varphi_v}')}{L_v(s,\Pi_v\times \widetilde{\pi}_v)L_v(w,\pi_v \times \pi_{1,v})},
\end{equation}
where $\Phi=\otimes \Phi_v\in\Pi$ is a cusp form and $s,w\in\BC$.

We may write
$$ H(\pi)=H(\Pi,\pi,\pi_1,s,w;\Fq,\Fl).$$

Using the notations in Section 3.4, we have
$$
\CM(H)=\CM(\Pi,\pi_1,s,w,\Fq,\Fl),
$$
where

$$
\CM(\Pi,\pi_1,s,w,\Fq,\Fl):=\CC(\Pi,\pi_1,s,w,\Fq,\Fl)+\CE(\Pi,\pi_1,s,w,\Fq,\Fl),
$$

with

$$
\CC(\Pi,\pi_1,s,w,\Fq,\Fl)= \CC(H)= \sum_{\substack{\pi\text{ cusp}^0\\\operatorname{cond}(\pi)\mid \Fq }}\frac{\Lambda(s,\Pi\times \widetilde{\pi})\Lambda(w,\pi \times \pi_1)}{\Lambda(1,\operatorname{Ad},\pi)} H(\Pi,\pi,\pi_1,s,w;\Fq,\Fl),
$$
which is the cuspidal contribution.

And

$$
\CE(\Pi,\pi_1,s,w,\Fq,\Fl)= \CE(H)= \int_{\substack{\pi\text{ eisen}^0\\\operatorname{cond}(\pi)\mid \Fq }}\frac{\Lambda(s,\Pi\times \widetilde{\pi})\Lambda(w,\pi \times \pi_1)}{\Lambda^*(1,\operatorname{Ad},\pi)} H(\Pi,\pi,\pi_1,s,w;\Fq,\Fl) d \mu_{\text{aut}}(\pi)
$$
which is the continuous (Eisenstein) contribution.

The notation $\text{cusp}^0$ and $\text{eisen}^0$ means that we are restricting to irreducible generic automorphic forms which are unramified at every archimedean place. We can define
$$
\CN(\Pi,\pi_1,s,w,\Fq,\Fl):=\CR(\check{H})-\CR(H).
$$

From the above discussion in Section 3, we can give the statement of our main theorem finally.

\begin{thm}\label{spectral-reciprocity}
Let $\Pi$ and $\pi_1$ be everywhere unramified cuspidal automorphic representation for $\GL(n+1)$ and $\GL(n-1)$ over $F$ with trivial central character.  Let $\widetilde{\pi}$ be the contragredient representation of $\pi$. Suppose that $\Fq$ and $\Fl$ are unramified, coprime ideals. Futhermore, we assume that $\frac{1}{2} \leq \Re(s),\Re(w) < \frac{n+1}{n+2}$. Then we will have the following identity
$$
\CM(\Pi,\pi_1,s,w,\Fq,\Fl)=\CN(\Pi,\pi_1,s,w,\Fq,\Fl)+\CM(\Pi,\pi_1,s',w',\Fl,\Fq),
$$
where the complex parameters $s',w'$ satisfy the relation $$s' = \frac{1+(n-1)w-s}{n}, \quad w' = \frac{(n+1)s+w-1}{n}.$$ 

\end{thm}

\end{document}